\documentclass[12pt]{article}
\usepackage[centertags]{amsmath}
\usepackage{pdfsync}
\usepackage{amsfonts}
\usepackage{amssymb}
\usepackage{amsthm}
\usepackage{amscd}
\usepackage{color}
\usepackage[textsize=small]{todonotes}
\usepackage{cases}
\usepackage{verbatim}
\usepackage[mathscr]{euscript}
\usepackage{graphicx}
\usepackage{bbm}
\usepackage{epsfig}
\usepackage{fancyhdr}
\usepackage{hyperref}
\usepackage{wrapfig}
\usepackage{hhline}
\usepackage{array}
\usepackage{tabularx,colortbl}

\theoremstyle{plain}
\newtheorem{thm}{Theorem}[section]

\newtheorem{prop}[thm]{Proposition}
\theoremstyle{definition}
\newtheorem{defn}[thm]{Definition}
\newtheorem{rem}[thm]{Remark}
\newtheorem{ex}[thm]{Example}

\newcounter{thestep}
{\medskip\refstepcounter{thestep}{{\textit{Step~\arabic{thestep}.}}}\begin{itshape}}
{\end{itshape}\medskip }

\newcounter{thehypothesis}
{\refstepcounter{thehypothesis}\begin{itemize}\item[(H\arabic{thehypothesis})]}
{\end{itemize}}

\newcounter{theproperty}
{\refstepcounter{theproperty}\begin{itemize}\item[(P\arabic{theproperty})]}
{\end{itemize}}

\newcommand{\compcent}[1]{\vcenter{\hbox{$#1\circ$}}}
\newcommand{\comp}{\mathbin{\mathchoice {\compcent\scriptstyle}{\compcent\scriptstyle}
{\compcent\scriptscriptstyle}{\compcent\scriptscriptstyle}}}

\newcommand{\om}{\omega}
\newcommand{\Om}{\Omega}

\renewcommand{\a}{\alpha}
\renewcommand{\b}{\beta}

\renewcommand{\hom}[3]{ {\mathit{Hom}_{\mathbb{#1}} \left ( {#2} , {#3}\right )} }

\newcommand{\im}[1]{{\mathrm{Im}\,  #1  }}
\newcommand{\re}[1]{{\mathrm{Re}\,  #1  }}
\newcommand{\ImH}{{ \mathrm{Im}\,\mathbb H }}

\newcommand{\hkred}{{/\!\! /\!\! /}}

\newcommand{\C}{{\mathbb{C}}}
\newcommand{\R}{{\mathbb{R}}}
\newcommand{\Z}{{\mathbb{Z}}}
\renewcommand{\H}{{\mathbb{H}}}
\newcommand{\bM}{{\mathbb M}}
\newcommand{\Mcasd}{{\mathcal M_{casd}}}

\DeclareMathOperator{\sign}{sign}

\DeclareMathOperator{\dirac}{{\mathcal D}}

\newcommand{\Comment}[2][\empty]{\ifthenelse{\equal{#1}{\empty}}{\todo[color=blue!15]{#2}}{\todo[color=blue!15,#1]{#2}}}

\newcommand{\hK}{{hyperK{\"a}hler }}
\newcommand{\gSW}{{generalized Seiberg--Witten }}

\begin{document}

\title{Dirac operators in gauge theory}%
\author{Andriy Haydys\\%
\textit{University of Bielefeld}\\%
}%

\date{November 5, 2013}


\maketitle

\begin{abstract}
This paper is a mixture of expository material and current research material. Among new results are examples of generalized harmonic spinors and their gauged version, the \gSW equations.
\end{abstract}

\section{Introduction}

A lot of advances in geometry and topology of low dimensional manifolds are intimately related to gauge theory. Recently a lot of interest attracted a variant of the anti-self-duality theory for higher dimensional manifolds equipped with metrics with special holonomies~\cite{DonaldsonThomas:98}.  The anti-self-dual (asd) instantons on such manifolds can blow-up along certain subspaces of co-dimension four~\cite{Tian:00}. It is argued in~\cite{DonaldsonSegal:09} and~\cite{Haydys:12_GaugeTheory_jlms} that in the limit one obtains certain ``generalized harmonic spinors'', which are harmonic sections of fiber bundles, the fibers of which are diffeomorphic to the moduli space of asd instantons on $\R^4$. On of the purposes of this paper is to present examples of generalized harmonic spinors and their gauged version, namely the \gSW equations.   

The exposition is chosen so that the material should be accessible for a reader not familiar with the basics of spin geometry or the Seiberg--Witten theory. Throughout the focus is on examples.

\medskip

Section~\ref{Sect_ClassDirac} is a rapid introduction to Dirac operators and the Seiberg--Witten theory. We briefly introduce Dirac operators focusing on dimension four, which has the advantage that the corresponding spin group can be constructed ``from scratch''. The spin groups in low dimensions are intimately related to quaternions and the language of quaternions is emphasized throughout. This has a twofold purpose. First, for an unprepared reader this is a quick way to understand what Dirac operators are at least in low dimensions. Secondly, this prepares the ground for a certain generalization of  Dirac operators considered later on. We finish the first part with the Seiberg--Witten equations emphasizing again the role of quaternions.   

In Section~\ref{Sect_genDirac} we introduce the generalized Dirac operator, whose zeros are the generalized harmonic spinors mentioned above. The generalized Dirac operator has its origins in physics~\cite{AnselmiFre:95} and was later considered also in mathematical literature~\cite{Taubes:99,Pidstrygach:04,  Haydys_ahol:08}. The idea of the generalization is very simple. Recall that for the Euclidean space $\R^4$ the Dirac equation can be written in the form
\begin{equation}\label{Eq_DiracR4}
\frac {\partial u}{\partial x_0}- i\frac {\partial u}{\partial x_1}- j\frac {\partial u}{\partial x_2} - k\frac {\partial u}{\partial x_3}=0,\qquad u\colon \R^4\to \H.
\end{equation}
This clearly generalizes for maps $u\colon \R^4\to M$ provided $M$ is a hypercomplex manifold. A generalization of~\eqref{Eq_DiracR4} for non-flat source manifolds requires either some further restrictions on the target $M$ or some  additional structure on the source manifold. In this paper the first possibility is pursued, while a realization of the second one can be found for instance  in~\cite{Salamon_hKFloer:08}.

In Section~\ref{Sect_gSW} examples of the \gSW equations are presented. With a suitable choice of the target space, these equations make sense for any four-manifold (or three-manifold). However, to omit  technical details, only the case of $\R^4$ as the source manifold is considered.  

\medskip
\textsc{Acknowledgement.} I am thankful to an anonymous referee for helpful comments.

\section{Dirac operators and the Seiberg--Witten equations}\label{Sect_ClassDirac}

\subsection{Clifford algebras and spin groups in low dimensions}
The purpose of this subsection is to recall briefly the notions of Clifford algebra and spin group focusing on low dimensions.  More details can be found for instance in~\cite{LawsonMichelsohn:89}. 

Since $\pi_1(SO(n))\cong\Z/2\Z$ for any $n\ge 3$, there is a simply connected Lie group denoted by $Spin(n)$ together with a homomorphism $Spin(n)\to SO(n)$, which is a double covering. This characterizes $Spin(n)$ up to an isomorphism. The spin groups can be constructed explicitly with the help of Clifford algebras, however in low dimensions this can be done using the quaternions only. This is the approach taken here.

Let $\H$ denote the algebra of quaternions. Denote
\[
Sp(1)=\{q\in\H\mid |q|=1 \}\cong S^3.
\]
Clearly, this is a simply connected Lie group. Furthermore, identify $\ImH=\{ \bar h=- h \}$ with $\R^3$ and consider the homomorphism
\begin{equation}\label{Eq_Sp1ToSO3}
\a\colon Sp(1)\to SO(3),\qquad q\mapsto A_q,
\end{equation}
where $A_q h=qh\bar q$. It is easy to check that the corresponding Lie-algebra homomorphism is in fact an isomorphism. Since $SO(3)$ is connected, $\a$ is surjective. Moreover,  $\ker\a = \{\pm 1\}$. Hence, \eqref{Eq_Sp1ToSO3} is a non-trivial double covering, i.e., $Spin(3)\cong Sp(1)$.   


\smallskip

To construct the group $Spin(4)$, first recall that the Hodge operator $*$ yields the splitting $\Lambda^2(\R^4)^*=\Lambda_+^2(\R^4)^*\oplus \Lambda^2_-(\R^4)^*$, where $\Lambda^2_\pm(\R^4)^*=\{\om\mid *\om=\pm\om\}$. Since $\mathfrak{so}(4)\cong \Lambda^2(\R^4)^*=\Lambda_+^2(\R^4)^*\oplus \Lambda^2_-(\R^4)^*=\mathfrak{so}(3)\oplus \mathfrak{so}(3)$, the adjoint representation yields a homomorphism $SO(4)\to SO(3)\times SO(3)$.

Identify $\R^4$ with $\H$ and consider the homomorphism\footnote{We adopt the common convention $Sp_\pm(1)=Sp(1)$. The significance of the subscripts ``$\pm$'' will be clear below.}
\[
\b\colon Sp_+(1)\times Sp_-(1)\to SO(4),\qquad (q_+, q_-)\mapsto A_{q_+\!,\, q_-},
\]
where $A_{q_+\!,\, q_-}h= q_+h \bar q_-$. An explicit computation shows that the composition $Sp_+(1)\times Sp_-(1)\to SO(4)\to SO(3)\times SO(3)$ is given by $(q_+, q_-)\mapsto (A_{q_+}, A_{q_-})$. Hence, the Lie algebra homomorphism corresponding to $\b$ is an isomorphism and $\ker\b$ is contained in $\{(\pm 1, \pm 1)\}$. As it is readily checked, $\ker\b=\{ \pm(1,1)\}\cong \Z/2\Z$. Hence, $Sp_+(1)\times Sp_-(1)\cong Spin(4)$.

\medskip

Let $U$ be an Euclidean vector space. Then the Clifford algebra $Cl(U)$ is the tensor algebra $TU=\R\oplus U\oplus U\otimes U\oplus\dots$ modulo the ideal generated by elements $u\otimes u +|u|^2\cdot 1$. In other words, $Cl(U)$ is generated by elements of $U$ subject to the relations $u\cdot u= -|u|^2$. For instance, $Cl(\R^1)\cong \R[x]/(x^2+1)\cong\C$. The algebra $Cl(\R^2)$ is generated by $1, e_1, e_2$ subject to the relations $e_1^2=-1=e_2^2$ and $e_1\cdot e_2=-e_2\cdot e_1$, which follows from $(e_1+e_2)^2=-2$. In other words,  $Cl(\R^2)\cong\H$. In general, $Cl(\R^n)$ is generated by $1,e_1,\dots, e_n$ subject to the relations $e_i^2=-1$ and $e_i\cdot e_j=-e_j\cdot e_i$ for $i\neq j$.

It is convenient to have some examples of modules over Clifford algebras. Such module is given by a vector space $V$ together with a map
\[
U\otimes V\to V, \qquad u\otimes v\mapsto u\cdot v,
\] 
which satisfies $u\cdot (u\cdot v)= -|u|^2 v$ for all $u\in U$ and $v\in V$. An example of a $Cl(U)$--module is $V=\Lambda U^*$, where the $Cl(U)$--module structure is given by the map
\[
u\otimes \varphi\mapsto \imath_u\varphi - \langle u,\cdot\rangle\wedge\varphi.
\]

Let $V$ be a quaternionic vector space. Then the quaternionic multiplication gives rise to  the map $\ImH\otimes V\to V$, $h\otimes v\mapsto h\cdot v$, which satisfies $h\cdot (h\cdot v)=-h\bar h v=-|h|^2v$. Thus any quaternionic vector space is a $Cl(\R^3)$--module. In particular, the fundamental representation $W\cong\H$ of $Sp(1)\cong Spin(3)$ with the action given by the left multiplication is a $Cl(\R^3)$--module.

Similarly, for any quaternionic vector space $V$ the space $V\oplus V$ is a $Cl(\R^4)$--module. Indeed, the $Cl(\R^4)$--module structure is induced by the map 
\begin{equation}\label{Eq_ClInDim4}
\H\otimes_\R (V\oplus V)\to V\oplus V,\qquad
h\otimes (v_1, v_2)\mapsto (h v_2,-\bar hv_1)=
\begin{pmatrix}
  0 &  h\\
 -\bar h & 0
\end{pmatrix}
\begin{pmatrix}
  v_1\\ v_2
\end{pmatrix}.
\end{equation}
In particular, the $Sp_+(1)\times Sp_-(1)$--representation $W^+\oplus W^-$ is a $Cl(\R^4)$--module. Here, as the notation suggests, $W^\pm$ is the fundamental representation of $Sp_\pm(1)$.

\subsection{Dirac operators}
Let $X$ be a Riemannian manifold. Denote by $Cl(X)\to X$ the bundle, whose fiber at a point $x\in X$ is $Cl(T_xX)\cong Cl(T_x^*X)$. Let $E\to X$ be a bundle of $Cl(X)$--modules, i.e., there is a morphism of vector bundles 
\[
Cl\colon TX\otimes E\to E,\qquad (v,e)\mapsto v\cdot e,
\]
such that $v\cdot (v\cdot e)=-|v|^2 e$. Then $E$ is called a Dirac bundle if it is equipped with an Euclidean scalar product and a compatible connection $\nabla$ such that the following conditions hold:
\begin{itemize}
\item $\langle v\cdot e_1, v\cdot e_2 \rangle= |v|^2\langle e_1, e_2 \rangle$ for any $v\in T_xX$ and $e_1, e_2\in E_x$;
\item $\nabla (\varphi\cdot s)= (\nabla \varphi)\cdot s + \varphi\cdot\nabla s$ for any $\varphi\in\Gamma(Cl(X))$ and $s\in\Gamma(E)$. 
\end{itemize}
Here $\nabla\varphi$ is obtained by extending the Levi--Civita connection from $TX$. 
\begin{defn}
If $E$ is a Dirac bundle, the operator
\[
\dirac\colon \Gamma(E)\xrightarrow{\ \nabla\ }\Gamma(T^*X\otimes E)\xrightarrow{\ Cl\ }\Gamma(E)
\] 
is called the Dirac operator of $E$. 
\end{defn}

The Dirac operator is a (formally) self-adjoint first order elliptic partial differential operator. Elements of $\ker\dirac$ are called harmonic.  An example of a Dirac operator  is given by choosing $E=\Lambda T^*X$, for which the corresponding Dirac operator is $\dirac= d+ \delta$~\cite[Thm~5.12]{LawsonMichelsohn:89}, where $\delta =\pm * d*$ and the sign depends on the dimension of the manifold and the degree of a form.

\medskip

For the sake of simplicity, let us focus on a low dimensional case, say dimension four. Thus, from now on $X$ denotes a Riemannian four-manifold. It is also convenient to assume that $X$ is oriented.  

As already mentioned above, the space $\H\oplus \H$ is a $Cl(\R^4)$--module. There are at least three ways to construct a Dirac bundle from this $Cl(\R^4)$--module.

One way is as follows. Denote by $P_{SO}\to X$  the $SO(4)$--bundle of oriented orthonormal frames. Then $X$ is called spin, if there is a $Spin(4)$--bundle $P_{Spin}\to X$, which is a fiberwise double covering of $P_{SO}$. Assume $X$ is spin and choose a spin structure. Then, considering $\H\oplus \H$ as the representation $W^+\oplus W^-$, one obtains the associated bundle still denoted by  $W^+\oplus W^-$. This bundle is called the spinor bundle of $X$ and its sections are called spinors.  The spinor bundle is a Dirac bundle and the corresponding Dirac operator is of the form
\[
\dirac=
\begin{pmatrix}
  0 & \dirac^-\\
 \dirac^+ & 0
\end{pmatrix},
\qquad \dirac^\pm\colon\Gamma(W^\pm)\to\Gamma(W^\mp).
\] 
The components  $\dirac^\pm$ are called Dirac operators too. For instance, in the case  $X=\R^4$ it readily follows from~\eqref{Eq_ClInDim4} that $\dirac^\pm\colon C^\infty(\R^4;\H)\to C^\infty(\R^4;\H)$ can be written as
\[
 \dirac^+ = -\frac \partial{\partial x_0} + i\frac \partial{\partial x_1} +  j\frac \partial{\partial x_2} + k\frac \partial{\partial x_3},\qquad 
\dirac^- =\frac \partial{\partial x_0} + i\frac \partial{\partial x_1} +  j\frac \partial{\partial x_2} + k\frac \partial{\partial x_3}.
\]

The second way is a slight modification of the first one. Namely, denote by $Spin^c(4)=(Spin(4)\times S^1)/\pm 1$ and consider a $Spin^c(4)$--bundle $P_{Spin^c}\to X$ such that $P_{Spin^c}/S^1\cong P_{SO}$ (this isomorphism is fixed throughout). Unlike spin structures, spin$^{\mathrm c}$ structures always exist on oriented four-manifolds. Furthermore, denote by $P_{det}$ the determinant bundle $P_{Spin^c}/SO(4)$, which is a principal $S^1$--bundle. By observing that $P_{Spin^c}$ is a double covering of $P_{SO}\times P_{det}$, we obtain that a choice of connection $a$ on $P_{det}$ together with the Levi--Civita connection on  $P_{SO}$ induces a connection on $P_{Spin^c}$.  Letting $Spin^c(4)$ act on $\H\oplus\H$ via
\[
[q_+, q_-, z]\cdot (h_1, h_2)=(q_+h_1\bar z, q_-h_2\bar z)
\]
we obtain the associated bundle still denoted by $W^+\oplus W^-$, which is again a Dirac bundle. Hence, similarly to the case of spin four-manifolds, we obtain the spin$^{\mathrm c}$-Dirac operators $\dirac_a^\pm\colon \Gamma(W^\pm)\to \Gamma(W^\mp)$.  

The third way is to view $\H\oplus\H$ as the $SO(4)$--representation $(\Lambda^2_+(\R^4)^*\oplus \R)\oplus \R^4$. This leads to the following Dirac operator
\[
\dirac=
\begin{pmatrix}
  0 & \dirac''\\
 \dirac' & 0
\end{pmatrix},
\qquad 
\begin{aligned}
  &\dirac'=\delta^+ +d\colon  \Om^2_+(X)\oplus\Om^0(X)\to \Om^1(X);\\
  &\dirac''=d^++\delta\colon \Om^1(X)\to \Om^2_+(X)\oplus\Om^0(X),
\end{aligned}
\] 
where $\delta^+$ is the restriction of $\delta=*d*$ to $\Om^2_+(X)$.

\medskip


We would like to mention briefly some applications of Dirac operators. An important property of Dirac operators is the Weitzenb\"ock formula (also known as the Bochner identity)~\cite[Thm~8.2]{LawsonMichelsohn:89}, which can be written in the form
\[
\dirac^2=\nabla^*\nabla +\mathcal R,
\]
where $\nabla^*\nabla$ is the connection Laplacian and $\mathcal R$ is an expression, which depends algebraically on the curvature tensor. For instance, in the case $D=d+\delta$ the restriction of $\mathcal R$ to $T^*X$ can be identified with the Ricci curvature. This implies in particular that for compact manifolds admitting  a metric with positive Ricci curvature the first Betti number vanishes.

For the spin-Dirac operator (not necessarily in dimension four) the curvature term $\mathcal R$ equals up to a constant to the scalar curvature. This implies that for a metric with positive scalar curvature there are no harmonic spinors. A consequence of this is that the signature of a spin  four-manifold, which admits a metric with positive scalar curvature,  vanishes (see, for instance,~\cite{Moore:96}). 

\medskip

A particular class of harmonic spinors is given by covariantly constant spinors (here, clearly,  the dimension of the base manifold does not need to be four). For  a complete simply connected irreducible manifold, the existence of a covariantly constant spinor is equivalent~\cite{Wang:89_ParallelSpinors} to the holonomy group being  one of  $SU(n)$, $Sp(n)$, $G_2$ or $Spin(7)$. A particularly interesting case for us is the last holonomy group, which can occur  on eight-manifolds only.  An eight-manifold $M$ with holonomy in $Spin(7)$ is characterized by the existence of a particular closed 4--form $\Om$ called the Cayley form. A 4-dimensional submanifold $X\subset M$ is called a Cayley-submanifold, if $\Om$ restricted to $X$ yields the volume form of the induced metric on $X$. Cayley-submanifolds are examples of calibrated submanifolds~\cite{HarveyLawson:82}, hence they are volume minimizing in their homology classes. Cayley-submanifolds, possibly singular, also arise as blow-up loci~\cite{Tian:00} of $Spin(7)$-instantons, which are discussed in some details in Section~\ref{Subsect_Spin7Instantons} below.  Finally, the space of infinitesimal deformations of Cayley-submanifolds is the space of harmonic spinors~\cite{McLean:98}.

\subsection{The Seiberg--Witten equations}

Consider the map
\[
\sigma\colon\H\to\ImH,\qquad x\mapsto xi\bar x.
\]
Putting $x=z+ jw$, this map can be written in a more common way, namely
\[
\C^2\to \mathfrak{su}(2),\qquad (z,w)\mapsto  \frac 12
\begin{pmatrix}
 |z|^2-|w|^2 & z\bar w\\
\bar z w & |w|^2- |z|^2
\end{pmatrix}.
\] 
The map $\sigma$ is $Spin^c(4)$--equivariant, if the source is regarded as the $W^+$--representation and the target as $\Lambda^2_+(\R^4)^*\cong \mathfrak{sp}_+(1)$. Choosing a spin$^{\mathrm c}$ structure $P_{Spin^c}$ on an oriented Riemannian four-manifold $X$, we obtain an induced map, still denoted by $\sigma$, between the associated fiber bundles:
\[
\sigma\colon W^+\to\Lambda^2_+T^*X.
\]
The Seiberg--Witten equations~\cite{SeibergWitten:94} are
\begin{equation}\label{Eq_SW}
  \begin{aligned}
    &\dirac_a^+\varphi=0,\\
    &F_a^+=\sigma(\varphi),
  \end{aligned}
\qquad (a,\varphi)\in\mathcal A(P_{det})\times\Gamma(W^+),
\end{equation}
where $\mathcal A(P_{det})$ is the space of all connections on $P_{det}$. The space of solutions is invariant under the action of the gauge group $\mathcal G=\{g\colon X\to S^1 \}$, which acts on $a$ by the gauge transformations and on $\varphi$ by the multiplication. 

From now on assume that $X$ is closed. Perturbing the second equation by a self-dual 2-form one can achieve that for a generic choice of such perturbation the moduli space of solutions
\[
\mathcal M_{SW}=\{(a,\varphi)\mid  (a,\varphi)\ \text{satisfies }\eqref{Eq_SW} \}/\mathcal G
\]
is a smooth oriented\footnote{This requires a choice of orientation on certain homology groups of $X$.} compact manifold of dimension $d=(c_1(P_{det})^2 - 2\chi(X)-3\sign (X))/4$, where $\chi$ and $\sign$ denote the Euler characteristic and the signature respectively.   

Choose a basepoint $x_0\in X$ and denote $\mathcal G_0=\{g\in\mathcal G\mid g(x_0)=1 \}$. Then the space $\{(a,\varphi)\ \text{satisfies }\eqref{Eq_SW}  \}/\mathcal G_0$ is a principal $S^1$--bundle over $\mathcal M_{SW}$. Let $\eta$ denote the first Chern class of this bundle. Then, for a given spin$^{\mathrm c}$ structure the Seiberg--Witten invariant is the integer $\langle \eta^{\frac d2},[\mathcal M_{SW}]\rangle$ provided $d$ is even and $0$ otherwise. This integer does not depend on the choice of the perturbation provided $b_2^+(X)>1$. Thus, the Seiberg--Witten invariant is an integer-valued function on the space of all  spin$^{\mathrm c}$ structures. Observe also that this space is an $H^2(X;\Z)$--torseur.

A reader who wishes to learn more about basics of Seiberg--Witten theory is encouraged to consult~\cite{Moore:96,Morgan:96,Marcolli:99, Nicolaescu:00}.
\medskip

Let us give some sample applications of the Seiberg--Witten theory. The Seiberg--Witten invariant does not vanish on symplectic four-manifolds with $b_2^+>1$.  Moreover for such manifolds the Seiberg--Witten invariant coincides with a variant of the Gromov--Witten invariant~\cite{Taubes00_SWGW}.  Hence, there are strong restrictions on the smooth type of four-manifolds admitting symplectic structures. 

The Seiberg--Witten invariant vanishes on connected sums of four-manifolds with $b_2^+>1$. There is however a refinement of the Seiberg--Witten invariant~\cite{BauerFuruta04_StableCohomRefinement} not necessarily vanishing on connected sums. This refinement is based on the Seiberg--Witten \emph{map} rather than on its zeros only. The Seiberg--Witten map was also used to prove~\cite{Furuta01_118Conj} the ``10/8--theorem'', which is a strong restriction on the intersection form of a smooth four-manifold. 

Seiberg--Witten theory also applies to the Riemannian geometry of four-manifolds. The Weitzenb\"ock formula can be used to show that on Riemannian four-manifolds with positive scalar curvature the Seiberg--Witten invariant vanishes~\cite{Witten:94}. Using Seiberg--Witten theory, LeBrun~\cite{LeBrun96_4MfldsWithoutEinsteinM} showed that on four-manifolds there are obstructions to the existence of Einstein metrics besides the Hitchin--Thorpe inequality.


With the help of the three-dimensional variant of the Seiberg--Witten equations, Taubes proved~\cite{Taubes07_WeinsteinConj}  the Weinstein conjecture for three-manifolds. This states that for any closed three-manifold equipped with a contact form the associated Reeb vector field has a closed orbit.

\section{Generalized Dirac operators}\label{Sect_genDirac}

\subsection{Generalized Dirac operators on four--manifolds}

Let $(U,i_1, i_2, i_3)$ be a quaternionic vector space, where $i_1, i_2,$ and $i_3$ are complex structures satisfying quaternionic relations. Let $(V, I_1, I_2, I_3)$ be another  quaternionic vector space. Considering  $(U, i_1)$ and $(V, I_1)$ as complex vector spaces, we can decompose the space $\hom\R UV$ into two components consisting of complex linear and complex antilinear maps. Bringing $i_2$ and $I_2$ into the consideration, we obtain that each component splits in turn into two subspaces (the remaining complex structures do not refine this splitting, since they are determined by the first two). This is shown schematically on the diagram: 
\vskip0.3cm
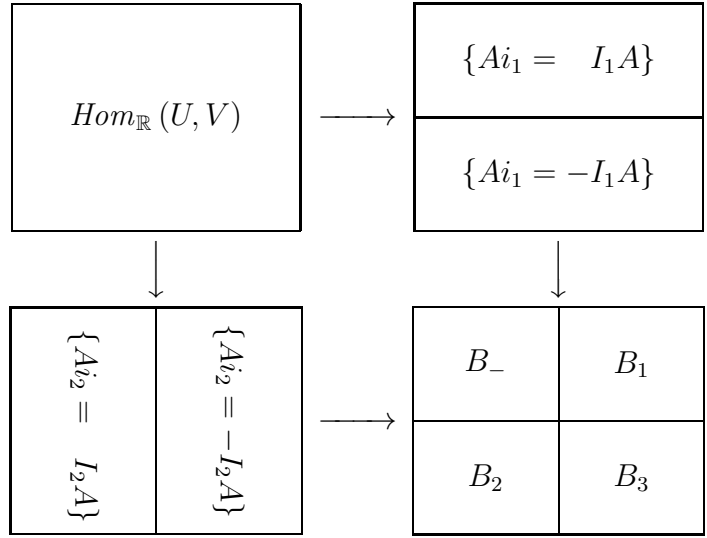
\[
\begin{gathered}
    \begin{CD}
 \begin{tabular}{| c |}\hline
    \begin{minipage}[c][3cm]{3.4cm}
      \begin{center} $ \hom\R UV$\end{center} 
    \end{minipage} \tabularnewline\hline
 \end{tabular} @>>> 
 \begin{tabular}{| c |}\hline
\begin{minipage}[c][1.5cm]{3.4cm}
\begin{center}$\{Ai_1=\phantom{-}I_1A \}$\end{center}
\end{minipage}\tabularnewline  \hline
\begin{minipage}[c][1.5cm]{3.4cm}
  \begin{center} $\{Ai_1=-I_1A \}$ \end{center}
\end{minipage}\tabularnewline \hline
 \end{tabular} \\
@VVV @VVV\\
   \begin{tabular}{| c | c |}\hline
  \begin{minipage}[c][3cm]{1.5cm}
 \begin{center}\rotatebox{270}{ $\{Ai_2=\phantom{-}I_2A \}$}\end{center} 
\end{minipage} &
\begin{minipage}[c][3cm]{1.5cm}
 \begin{center} \rotatebox{270} {$\{Ai_2=-I_2A \}$} \end{center} 
\end{minipage} \tabularnewline \hline
\end{tabular}
@>>> 
\begin{tabular}{| c | c |}\hline
\begin{minipage}[c][1.5cm]{1.5cm}\begin{center}{ $B_-$}\end{center} \end{minipage}&  
\begin{minipage}[c][1.5cm]{1.5cm}\begin{center}{ $B_1$}\end{center} \end{minipage} \tabularnewline \hline
\begin{minipage}[c][1.5cm]{1.5cm}\begin{center}{ $B_2$}\end{center} \end{minipage} & 
\begin{minipage}[c][1.5cm]{1.5cm}\begin{center}{ $B_3$}\end{center} \end{minipage}   \tabularnewline \hline
\end{tabular}
    \end{CD}
\end{gathered}
\]\\
Here $B_-=\hom{\mathbb H} UV$ and $B_j=\{ A\mid Ai_j=I_jA,\ Ai_k=-I_kA\text{ for}\  k\neq j\}$.

Notice that the group $SO(3)$ acts on the space of quaternionic structures of any quaternionic vector space. Hence there is an induced action of $SO(3)\times SO(3)$ on $\hom \R UV$. Consider the action of the diagonally embedded $SO(3)$. Then $B_-$ is the trivial representation. Though each individual subspace $B_j,\; j=1,2,3,$ is not preserved by this action, their direct sum $B_1\oplus B_2\oplus B_3=B_+$ is. To summarize, we have a splitting 
\begin{equation}\label{Eq_SplittingOfHomR}
\hom \R UV=B_-\oplus B_+=\hom{\mathbb H} UV\oplus B_+,
\end{equation}
which is invariant with respect to a simultaneous rotation of complex structures on both $U$ and $V$. 
\begin{rem}\label{Rem_CliffMultAndHomH}
It is easy to check that the map
\begin{equation}
A\mapsto \frac 14 \bigl (A-I_1Ai_1- I_2Ai_2-I_2Ai_3\bigr )
\end{equation}
is in fact the projection onto $\hom{\mathbb H} UV$. In particular, for $U=\H$ this projection 
can be written equivalently as 
\[
\hom R{\R^4}V\to V\cong \hom H{\H}V,\qquad A\mapsto  \frac 14 (Ae_0 -I_1Ae_1-I_2Ae_2 -I_3Ae_3). 
\]
This in turn can be identified with the map $\H\otimes_{\R}V\to V,\ h\otimes v\mapsto \bar h\cdot v$ (cf.~\eqref{Eq_ClInDim4}).
\end{rem}

 \medskip

Let $\tilde G$ be a Lie group together with a homomorphism $\Z/2\Z\to Z(\tilde G)$. Denote 
\[
G=\tilde G/\pm 1\quad\text{and} \quad \hat G=\bigl (Sp(1)\times\tilde G\bigr )/\pm 1.
\]
Clearly, there is a surjective Lie group homomorphism
\[
(\rho_1, \rho_2)\colon \hat G\longrightarrow SO(3)\times G
\]  
with a finite kernel.

Let $M$ be a hyperK{\"a}hler\footnote{The metric structure of $M$ is non-essential for the purposes of this section but will play a role below.} manifold. In particular, $M$ comes equipped with a triple $(I_1, I_2, I_3)$ of complex structures satisfying the quaternionic relations. Then for any purely imaginary quaternion $a=a_1i+a_2j+a_3k$ of unit length denote $I_a=a_1I_1 +a_2I_2+ a_3I_3$. We assume that $\hat G$ acts on $M$ such that the following conditions hold:
\begin{itemize}
\item[(i)] $(L_{\hat g})_* I_a (L_{\hat g^{-1}})_*= I_{\rho_1(\hat g)a}$, where $L_{\hat g}\colon M\to M,\ m\mapsto \hat g\cdot m$;
\item[(ii)] $\tilde G$  preserves the \hK structure of $M$, where $\tilde G$ is viewed as a subgroup of $\hat G$. 
\end{itemize}
In the sequel, $M$ is called the target manifold.

\medskip

Let $X^4$ be an oriented Riemannian four-manifold. 
 Denote by $P_\pm\xrightarrow{\; \pi_\pm\; }X$ the principal $SO(3)$--bundle of oriented orthonormal frames of $\Lambda_\pm^2T^*X$.  For any $x\in X$ a quaternionic structure $(i_1, i_2, i_3)$ on $T_xX$ compatible with the scalar product and the orientation gives rise to a frame $(\om_1^+, \om_2^+, \om_3^+)$ of $\Lambda^2_+T_x^*X$, where $\om_j^+=g_x(i_j\cdot, \cdot)$. This correspondence allows us to interpret a point $p_+\in P_+$ as a quaternionic structure on $T_xX$. 


Let $\hat P$ be a principal $\hat G$--bundle. Then $\hat P/ \tilde G$ is a principal $SO(3)$--bundle. We assume that $\hat P/\tilde G$ is isomorphic to $P_+$ and fix an isomorphism throughout. Similarly, $\hat P/Sp(1)=P_G$ is a principal $G$--bundle.  Moreover,  $\hat P$ is a finite covering of $P_+\times P_G$. Hence, a connection $a$ on $P_G$ together with the Levi--Civita connection on $P_+$ determines a connection $\hat a$ on $\hat P$.

\medskip 

Let $\bM=\hat P\times_{\hat G}M\xrightarrow{\ \pi\ }X$ be the associated bundle. Denote by $\mathcal V\rightarrow \bM$ the vertical tangent bundle, i.e., $\mathcal V=\ker \pi_*$. For any section $u$ of $\mathbb M$ the covariant derivative $\nabla^au$ is a section of $\hom R{TX}{u^*\mathcal V}$. A consequence of Property (i) above is that for each $x\in X$ the vector space $\mathcal V_{u(x)}$ has a distinguished $SO(3)$--worth of quaternionic structures $\mathcal I_{u(x)}$, which can be canonically identified with $P_{+,x}$. Recalling that splitting~\eqref{Eq_SplittingOfHomR} is $SO(3)$-invariant, we obtain
\[
\hom R{TX}{u^*\mathcal V}=\mathit{Hom}_{\H} ({TX}, {u^*\mathcal V})\oplus \mathit{Hom}^+_{\mathbb R} ({TX},{u^*\mathcal V}^+).
\]
\begin{defn}
  We call the map $\dirac_a\colon u\mapsto -4(\nabla^a u)_{\H}$ the generalized Dirac operator of $\bM$, where $(\nabla^a u)_{\H}$ is the $\H$-linear component of $\nabla^a u$. 
\end{defn}
\begin{ex}
Let $P_{Spin^c}\to X$ be a spin$^{\mathrm c}$ structure. Recall that there is a short exact sequence
\[
\{ 1\}\to Sp_-(1)\to Spin^c(4)\to \bigl ( Sp_+(1)\times S^1  \bigr )/\pm 1\to \{ 1\}.
\] 
Hence, $\hat P=P_{Spin^c}/Sp_-(1)$ is a principal $\hat G$--bundle, where $\hat G= \bigl ( Sp_+(1)\times S^1  \bigr )/\pm 1\cong U(2)$.  Furthermore, put $G=S^1$ and $M=\H$. Here $\H$ is viewed as being equipped with its left quaternionic structure and the $\hat G$--action is given by $[q_+, z]\cdot h= q_+h\bar z$. It follows from Remark~\ref{Rem_CliffMultAndHomH} that for these choices the ``generalized'' Dirac operator equals the spin$^{\mathrm c}$-Dirac operator $\dirac_a^+$, where $a$ is a connection on the $S^1$--bundle $\hat P/Sp_+(1)=P_{det}$.
\end{ex}

Suitably modifying  this example,  one can also obtain the spin-Dirac operator $\dirac^+$ and $\delta^+ +d$. Details are left to the reader.        

\begin{rem}
  In the case $X=\R^4,\ G=\{ 1\}$, and $M=\H$ the equation $\dirac u=0$ coincides with~\eqref{Eq_DiracR4}. This equation was  studied by Fueter~\cite{ Fueter:34} in his attempts to construct a quaternionic version of the theory of holomorphic functions. Therefore, sometimes generalized harmonic spinors are also called Fueter-sections.
\end{rem}



\subsection{Generalized Dirac operators on K{\"a}hler surfaces}\label{Subsect_DiracOnKaehler}
In this subsection a special case of the construction presented in the preceding  subsection is studied. First, only the case  $\hat G=SO(3)$  is considered here. Secondly, $X$ is assumed to be a K\"ahler surface. 

Notice that the above assumptions imply in particular that $\hat P=P_+$. Moreover, for a K\"ahler surface the structure group of $P_+$ reduces to $S^1$. Concretely, one can think of $S^1\subset SO(3)$ as a stabilizer of a non--zero vector, say $(1,0,0)\in\R^3$. Let $P_{red}\subset P_+$ denote the corresponding $S^1$--subbundle. With our choices $P_{red}$ is the principal bundle of the canonical bundle $K_X$. Then 
\begin{equation}
  \label{Eq_bMredStrGp}
  \bM= P_{red}\times_{S^1}M.
\end{equation}
The induced $S^1$--action on $M$ fixes $I_1$ and  rotates the other two complex structures.  This implies that the vertical bundle $\mathcal V$ carries a distinguished complex structure, which by a slight abuse of notation is also denoted by $I_1$. Then for $u\in\Gamma(\bM)$ we denote by $\partial u\in \Om^{1,0}(X;u^*\mathcal V)$ the $(1,0)$--component of the covariant derivative.   

The proof of the next theorem is  adapted from the proof of Proposition~4 in~\cite{Haydys_ahol:08}.

\begin{thm}
  Let $X$ be a compact K{\"a}hler surface. Then a spinor $u\in\Gamma(\bM)$ is harmonic, i.e., $\dirac u=0$, if and only if $\partial u=0$. 
\end{thm}
\begin{proof} First observe that the commutativity of $S^1$ together with \eqref{Eq_bMredStrGp} imply that $\bM$ carries a fiberwise action of $S^1$, which in turn induces an $S^1$--action on $\Gamma(\bM)$. 

For any $u\in\Gamma(\bM)$ we have the Weitzenb{\"o}ck--type formula~\cite{Pidstrygach:04}
\begin{equation}\label{Eq_WeitzOnKahler}
\| \dirac u\|^2 = \| \nabla u \|^2 +\int\limits_X\rho_0(u)\frac s4 \, vol_X,
\end{equation}
where $s$ is the scalar curvature of $X$ and $\rho_0$ is an $S^1$--invariant function. Here we used the fact that the self--dual part of the Weyl tensor vanishes on K{\"a}hler surfaces.  Since the right hand side of~\eqref{Eq_WeitzOnKahler} is $S^1$--invariant, for any harmonic spinor $u$ and any $z\in S^1$ the spinor $z\!\cdot\! u$ is also harmonic.

Pick a point $x\in X$ and choose a local trivialization of $P_{red}$ on a neighbourhood $W$ of $x$. This trivialization gives rise to an almost quaternionic structure $(i_1, i_2, i_3)$ on $W$, where  $i_1$ is in fact  the globally defined  complex structure of $X$. Let $u_{loc}\colon W\to M$ be the local representation of $u$. Then by Remark~\ref{Rem_CliffMultAndHomH} the harmonicity of $u$  yields:
\begin{equation}\label{Eq_HarmOfuLoc}
  \nabla_{ v}u_{loc}- I_1\nabla_{i_1v}u_{loc} - I_2\nabla_{i_2 v}u_{loc}- I_3\nabla_{i_3 v}u_{loc}=0.
\end{equation}
Here $v$ is an arbitrary local vector field on $W$. Substituting $z\!\cdot\! u_{loc}=L_zu_{loc}$ instead of $u_{loc}$ in~\eqref{Eq_HarmOfuLoc}, one obtains after a transformation:
\begin{equation}\label{Eq_HarmOfzuLoc}
\nabla_{ v}u_{loc}- (L_{\bar z})_*I_1(L_z)_*\nabla_{i_1v}u_{loc} - (L_{\bar z})_*I_2(L_z)_*\nabla_{i_2 v}u_{loc}- (L_{\bar z})_*I_3(L_z)_*\nabla_{i_3 v}u_{loc}=0.
\end{equation}

Recall that the $S^1$--action preserves $I_1$ and rotates the other two complex structures. Hence,  $(L_{\bar z})_*I_1(L_z)_*=I_1$ for any $z\in S^1$ and there exists some $z\in S^1$ such that  $(L_{\bar z})_*I_2(L_z)_*=-I_2$ and $(L_{\bar z})_*I_3(L_z)_*=-I_3$. Then~\eqref{Eq_HarmOfzuLoc} yields
\[
 \nabla_{ v}u_{loc}- I_1\nabla_{i_1v}u_{loc} +I_2\nabla_{i_2 v}u_{loc}+ I_3\nabla_{i_3 v}u_{loc}=0.
\]
Summing this with~\eqref{Eq_HarmOfuLoc} leads to 
\begin{equation}\label{Eq_DelLoc}
\nabla_{ v}u_{loc}- I_1\nabla_{i_1v}u_{loc} =0\quad \Longleftrightarrow\quad \nabla_{i_1v}u_{loc}=-I_1\nabla_{ v}u_{loc}, 
\end{equation}
which means $\partial u=0$. On the other hand, it is easy to see that~\eqref{Eq_DelLoc} implies~\eqref{Eq_HarmOfuLoc}. This finishes the proof.
\end{proof}

\paragraph{Holomorphic sections of bundles with fiber $T^*Gr_k(\C^n)$.}

Let $(M_1, J_1, J_2, J_3)$ be a \hK manifold equipped with an action of $S^1$, which fixes one complex structure, say $J_1$, and rotates the other two complex structures.\footnote{This action does not need to extend to an action of $SO(3)$ as above.} Assume also that there is an $S^1$--equivariant  $(J_1, I_1)$--anti\-ho\-lomorphic map $\tau\colon M_1\to M$. Then given a holomorphic section $u_1$ of $\bM_1=\hat P_{red}\times_{S^1}M_1$ we obtain a harmonic spinor $u\in\Gamma(\bM)$ by composing $u_1$ with the map $\bM_1\to\bM,\ [\hat p, m_1]\mapsto [\hat p,\tau (m_1)]$.  An example of this will be given below.

A large class of \hK manifolds $M_1$ admitting $S^1$--action as described above was constructed in~\cite{Kaledin:99,Feix:01}, where $M_1$ is the cotangent bundle of a K{\"a}hler manifold $Z$ (usually the \hK metric is defined only in some neighbourhood of the zero section). It is assumed in this case that $z\in S^1$ acts on $T^*Z$ by the multiplication by $z^p$ for some $p\in\mathbb Z$. 

Identify $u_1\in\Gamma(\bM_1)$ with an equivariant map $\hat u_1\colon  P_{red}\to M_1$. In the case $M_1=T^*Z$ the map $\hat u_1$ can be composed with the projection $T^*Z\to Z$. The result is an $S^1$--invariant map $P_{red}\to Z$ or, equivalently, a map $v\colon X\to Z$. Writing $T^*Z\cong T^*Z\otimes_{\C}\C$ and letting $S^1$ act on $\C$ only, we see that the lift of $v$ is given by a section $\psi$ of $v^*T^*Z\otimes K_X^p$.  Moreover, holomorphicity of $u$ is equivalent to the holomorphicity of both $v$ and $\psi$. Thus, in the case $M_1=T^*Z$ we have 
\[
\{u_1\in\Gamma(\bM_1)\mid \bar\partial u_1=0\}\cong \{ (v,\psi)\mid v\in Map(X,Z), \psi\in \Gamma (v^*T^*Z\otimes K_X^p),\ \bar\partial v=0,\ \bar\partial\psi=0\}.
\]

Let us consider the case $Z=Gr_n(\C^r)$ in some details. Recall that for a compact complex manifold $X$ any holomorphic map $v\colon X\to Gr_n(\C^r)$ arises from an $r$--dimensional subspace $V\subset H^0(X; E)$ for some rank $n$ holomorphic vector bundle $E\to X$ that is generated by holomorphic sections from $V$. Moreover, if $S\to Gr_n(\C^r)$ denotes the tautological vector bundle, then $E\cong v^*S$ and there is an embedding $E\hookrightarrow \underline\C^r$. Furthermore, $F=\underline\C^r/E$ is the pull-back of the canonical  factor bundle $Q$ on $Gr_n(\C^r)$. Since $T^*Gr_n(\C^r)\cong Q^\vee\otimes S$, it follows that $v^*T^*Gr_n(\C^r)\cong F^\vee\otimes E$. Hence we obtain the following result.

\begin{prop}\label{Prop_HolSectionsT*Gr}
  For a compact K\"ahler surface  $X$, any holomorphic section of $P_{red}\times_{S^1}T^*Gr_n(\C^r)\to X$ can be constructed from the following data:
  \begin{itemize}
  \item A holomorphic rank $n$ vector bundle $E$ admitting $r$ global holomorphic sections that generate $E$;
  \item A holomorphic section of $(\underline\C^r/E)^\vee\otimes E\otimes K_X^p$.\qed
  \end{itemize}
\end{prop}




\paragraph{Generalized harmonic spinors with values in the space of anti-self-dual instantons.} Denote by  $M_{n,r}(\mathbb K)$ the space  of matrices with $n$ rows and $r$ columns with entries from a ring $\mathbb K\in\{\R,\C,\H\}$. Consider the flat \hK manifold 
\[
N=M_{n,n}(\H)\oplus M_{n,r}(\H)\cong 
M_{n,n}(\C)\oplus M_{n,n}(\C)\oplus M_{n,r}(\C)\oplus M_{r,n}(\C).
\]
The group $U(n)$ acts on $N$ as follows: $(B_1, B_2, C, D)\cdot g= (g^{-1}B_1,\, g^{-1}B_2,\, g^{-1}C,\, Dg)$. The corresponding moment map $\mu\colon N\to \mathfrak{u}(n)\otimes\ImH,\ \mu=\mu_\R i+\mu_\C j$, is given by
\[
\mu_\R=\frac i2\bigl ( [B_1, \bar B_1^t]+ [B_2, \bar B_2^t] + C\bar C^t -\bar D^t D\bigr ),\qquad
\mu_\C = [B_1, B_2] +CD.
\]
Denote 
\[
\mathcal M_0(r,n)=\bigl \{  (B_1, B_2, C, D)\in N\mid \mu(B_1, B_2, C, D)=0  \bigr\}/U(n).
\]
This space is called the \hK reduction of $N$ and carries itself a \hK structure outside the singular locus~\cite{HitchinAO:87}. Moreover, by the ADHM construction~\cite{ADHM:78} there is a bijection between the non-singular part of $\mathcal M_0(r, n)$ and the  moduli space of framed asd connections on a Hermitian bundle $E\to S^4$ of rank $r$ and second Chern class $n$. Also, the following result holds.

\begin{thm}[{\cite[Cor.\, 3.4.10]{DonaldsonKronheimer:90}}]
  There is a bijection between $\mathcal M_0(r,n)$ and the moduli space of framed \emph{ideal} instantons on $\R^4$.
\end{thm}

Denote 
\[
\mathcal M_i(r,n)=\bigl \{  (B_1, B_2, C, D)\in N\mid \mu(B_1, B_2, C, D)=i  \bigr\}/U(n).
\]
Clearly, $\mathcal M_i(r,n)$ is also a \hK reduction of $N$ but with respect to a different value of the moment map. This is a non-singular \hK manifold, which was extensively studied by Nakajima~\cite{Nakajima:94, Nakajima_Resolutions:94,Nakajima:99}. In particular, $\mathcal M_i(r, n)$ is equipped with an $S^1$--action, which preserves one complex structure, say $J_1$, and rotates the other two complex structures. Moreover, there is a $(J_1, I_1)$--holomorphic map
\[
\pi\colon \mathcal M_i(r,n)\longrightarrow\mathcal M_0(r,n).
\]
By putting $B_1=0=B_2$ we see that  $M_i(r,n)$ contains  a \hK submanifold $M_{n,r}(\H)\hkred_{\mu=i}U(n)=\{\mu(0,0, C,D)=i\}/U(n)$, which is biholomorphic to $T^*Gr_n(\C^r)$ with respect to $J_1$ (details can be found for instance in~\cite[p.303]{haydys_hk:08}).  The complex conjugation on $M_{n,r}(\C)\oplus M_{r,n}(\C)\cong M_{n,r}(\H)$ induces a $J_1$--antiholomorphic map $c$ on $T^*Gr_n(\C^r)$. Hence we obtain a $(J_1, I_1)$--antiholomorphic map 
\[
\tau\colon T^*Gr_n(\C^r)\xrightarrow{\ c\ }T^*Gr_n(\C^r)
\hookrightarrow \mathcal M_i(r,n)\xrightarrow{\ \pi\ } \mathcal M_0(r,n).
\]
Letting $S^1$ act on $\mathcal M_i(r,n)$ via 
\[
z\cdot (B_1, B_2, C, D)= (B_1, z^p B_2, C, z^p D)
\]
the map $\tau$ becomes $S^1$--equivariant. Thus, recalling that composition of a holomorphic section  with $\tau$ results in a harmonic spinor, we obtain the following result.

\begin{thm}
Let $X$ be a compact K\"ahler surface. Then the same data as in Proposition~\ref{Prop_HolSectionsT*Gr} determine a harmonic section of $\bM_0(r,n)= P_{red}\times_{S^1}\mathcal M_0(r,n)$ at least away from the singular locus.\qed
\end{thm}


\section{Generalized Seiberg--Witten equations: examples}\label{Sect_gSW}

In this section a generalization of the Seiberg--Witten equations first introduced in~\cite{Taubes:99} and~\cite{Pidstrygach:04} is considered. This generalization makes sense for any four-manifold (or three-manifold) but for the sake of simplicity only the case $X=\R^4$ is considered here. 

On $\R^4$ the \gSW equations can be defined as follows. Let $M$ be  a \hK manifold  equipped with a tri-Hamiltonian action of a Lie group $G$. Assume that $\mathfrak g$ is endowed with an $ad$-invariant scalar product and denote by $\mu\colon M\to\mathfrak g\otimes \ImH$ the corresponding moment map. For  a pair $(u, a)\in C^\infty(\R^4; M)\times \Om^1(\R^4; \mathfrak g)$, where $a$ should be thought of as a connection on a trivial $G$--bundle,  consider the following equations
\begin{equation}\label{Eq_gSW}
  \dirac_a u=0,\qquad F_a^+=\mu\comp u, 
\end{equation}
which are called the \gSW equations. The first equation is already familiar from Section~\ref{Sect_genDirac}, while the second one needs a little explanation. Identifying $\Lambda^2_+(\R^4)^*$ with $\ImH$, one can interpret $\mu\comp u$ as a self-dual $2$-form on $\R^4$ with values in $\mathfrak g$. This matches the term on the left hand side.


A natural parameter of the construction is the target manifold $M$ together with the $G$--action. Gauge theories, which can be obtained for different choices of $M$, are considered in some details below. The simplest example is $M=\H$. For $G=S^1$, which acts by the multiplication  on the right, one recovers the classical Seiberg--Witten equations. The details are left to the reader.

\subsection{Vafa--Witten equations}

 Consider the flat \hK manifold $M=\mathfrak g\otimes\H$ as the target manifold. Let $G$ act on $M$ by the ``quaternization'' of the adjoint action. The corresponding moment map is given by
\begin{equation*}
  \begin{aligned}
  \mu(\xi) &=  \bigl ([\xi_2,\xi_3]+[\xi_0,\xi_1]\bigr )\otimes i +
                      \bigl ([\xi_3,\xi_1]+[\xi_0,\xi_2]\bigr )\otimes j + 
                     \bigl ([\xi_1,\xi_2]+[\xi_0,\xi_3]\bigr )\otimes k\\
                 &= \sigma(\im\xi)+[\re\xi,\im\xi],    
  \end{aligned}
\end{equation*}
where $\xi=\xi_0+(\xi_1i+\xi_2j+\xi_3k)= \re \xi +\im \xi$.

Furthermore, a map $u\colon \R^4\to \mathfrak g\otimes\H$ can be identified with a pair $(c,b)\in\Om^0(\R^4;\mathfrak g)\oplus \Om^2_+(\R^4;\mathfrak g)$, where $c$ is the real part of $u$ and $b$ is obtained from the imaginary part of $u$ via the identification $\ImH\cong \Lambda^2_+(\R^4)^*$. The corresponding Dirac operator is $\dirac_a(c,b)= d_ac+ \delta_a^+b$. Hence, in the case $M=\mathfrak g\otimes\H$, the \gSW equations are
\begin{equation*}
  \begin{aligned}
    &d_ac+ \delta_a^+ b=0,\\
    & F_a^+-\sigma(b)+[b,c]=0,
  \end{aligned}
\qquad (a,b,c)\in\Om^1(\R^4;\mathfrak g)\times \Om^1(\R^4;\mathfrak g)\times \Om^0(\R^4;\mathfrak g).
\end{equation*}
These equations first appeared in~\cite{VafaWitten:94_StrongCouplingTest} and are known as the Vafa--Witten equations. Notice that the Vafa--Witten equations make sense on any (oriented Riemannian) four-manifold. 

\subsection{Anti-self-duality equations on $G^\C$--bundles}

The previous example admits a different interpretation. Namely, the target manifold is chosen again to be $M=\mathfrak g\otimes \H$ but this time a map $u\colon \R^4\to \mathfrak g\otimes \H$  is identified with some $b\in\Om^1(\R^4;\mathfrak g)$ according to the rule
\[
u=\xi_0+\xi_1i+\xi_2j+\xi_3k\equiv \sum_{p=0}^3\xi_pdx_p=b. 
\]
The group $G$ acts on $M$ in the same manner as in the previous example and therefore the moment map is given by the same expression. However this is  also interpreted in a different way. Namely, for $b\in \Om^1(\R^4;\mathfrak g)$ a straightforward computation shows that the self-dual 2-form $\mu\comp b$ is in fact $\frac 12[b\wedge b]^+$, where the symbol $[\cdot\wedge\cdot]$ stays for a combination of the wedge product and the Lie--brackets. The Dirac operator acting on 1-forms was already mentioned above and equals $\delta_a + d_a^+$. 
 Hence, in this case the \gSW equations take the following form
  \begin{align}
    & \delta_a b=0,\label{Eq_cxASD1}\\
    & d_a^+b=0,\label{Eq_cxASD2}\\
    & F_a^+-\tfrac 12[b\wedge b]^+=0.\label{Eq_cxASD3}
  \end{align}

To see the geometric meaning of equations~\eqref{Eq_cxASD1}-\eqref{Eq_cxASD3}, consider the 1-form $A=a+ib\in\Om^1(\R^4;\mathfrak g^\C)$. Interpreting $A$ as a connection on a trivial $G^\C$--bundle, one obtains
\[
F_A^+=\bigl ( F_a + id_ab -\tfrac 12[b\wedge b] \bigr )^+= F_a^+-\tfrac 12[b\wedge b]^+ + id_a^+ b.
\]
Hence, Equations~\eqref{Eq_cxASD2} and~\eqref{Eq_cxASD3} mean that $A=a+ib$ is anti-self-dual. Notice that these equations are invariant with respect to the \emph{complex} gauge group $Map(\R^4; G^\C)$. 

It remains to clarify the meaning of~\eqref{Eq_cxASD1}. Notice that for a compact manifold $X$ the space $\Om^1(X;\mathfrak g^\C)$ has a natural K\"ahler metric, which is preserved by the action of the \emph{real} gauge group $Map(X; G)$. Then $(a,b)\mapsto \delta_a b$ is the moment map of this action. The same conclusion holds for $X=\R^4$ provided $\Om^1(X;\mathfrak g^\C)$ is replaced by a suitable Sobolev space. Thus, solutions of~\eqref{Eq_cxASD1}--\eqref{Eq_cxASD3} are those anti-self-dual connections, which are in the zero level set of the moment map of the real gauge group. In other words,~\eqref{Eq_cxASD1} is a ``stability condition''.

For the sake of brevity solutions of~\eqref{Eq_cxASD1}--\eqref{Eq_cxASD3} are called (stable) complex anti-self-dual connections. 
\bigskip

The moduli space of complex asd connections has some interesting properties, which are best seen from a more general perspective. For this reason it is convenient do deviate from the convention to work exclusively with $\R^4$ as the base four-manifold. Thus, let $X$ be a closed Riemannian oriented four-manifold. Choose a principal $G$--bundle $P\to X$ and denote $ad\, P=P\times_G\mathfrak g$.   Equations~\eqref{Eq_cxASD1}--\eqref{Eq_cxASD3} for a general four-manifold can be written in exactly the same form with the understanding that $a$ stays for a connection on $P$ and $b$ is a 1--form on $X$ with values in $ad\, P$. The geometric meaning remains also valid in this case if $A=a+ib$ is interpreted as a connection on the corresponding principal $G^\C$--bundle $\mathcal P=P\times_GG^\C$. 

\begin{rem}
  Strictly speaking, on general four-manifold the complex anti-self-duality equations do not quite fit into the concept of~\cite{Pidstrygach:04}. The reason is, roughly speaking, that $Sp_-(1)$ must act non-trivially on the target manifold.
\end{rem}

Let $\Mcasd (\mathcal P)$ denote the moduli space of complex asd connections. Clearly the moduli space of real asd connections $\mathcal M_{asd}(P)$ is contained in $\Mcasd (\mathcal P)$. By looking at the deformation complex of~\eqref{Eq_cxASD1}-\eqref{Eq_cxASD3}
\[
0\to\Om^0(ad\, P)\to \Om^1(ad\, P)\oplus \Om^1(ad\, P)\to \Om^0(ad\, P)\oplus \Om^2_+(ad\, P) \oplus \Om^2_+(ad\, P)\to 0
\]
it is easy to see that the expected dimension of $\Mcasd (\mathcal P)$ is twice the expected dimension of $\mathcal M_{asd}(P)$.

\begin{thm}\label{Thm_ModuliOfCasd}
  Let $X$ be a closed oriented Riemannian four-manifold. Assume that both  $\mathcal M_{asd}(P)$ and $\Mcasd (\mathcal P)$ are manifolds of expected dimensions. Then the following holds:
  \begin{itemize}
  \item[(i)] $\Mcasd (\mathcal P)$ is K\"ahler;
  \item[(ii)] $\mathcal M_{asd}(P)$ is a Lagrangian submanifold of $\Mcasd (\mathcal P)$;
  \item[(iii)] If $X$ is K\"ahler, then  $\Mcasd (\mathcal P)$ is \hK and  $\mathcal M_{asd}(P)$ is a complex Lagrangian submanifold.
  \end{itemize}
\end{thm}

The rest of this subsection is devoted to the sketch of the proof of Theorem~\ref{Thm_ModuliOfCasd}. The configuration space $\mathcal A(P)\times \Om^1(ad\, P)\cong T^*\mathcal A(P)$ is a flat infinite dimensional K\"ahler manifold.\footnote{Strictly speaking, one should pass to a suitable Sobolev space to get a Banach manifold structure; Here and in the sequel we work in a smooth category for the sake of simplicity of exposition.} Indeed, the K\"ahler structure is given explicitly by
\begin{equation}\label{Eq_KaehlerStr}
I_1 (\mathrm v,\mathrm w)=(-\mathrm w,\mathrm v),\quad 
\om_1\bigl ((\mathrm v_1, \mathrm w_1), (\mathrm v_2, \mathrm w_2)\bigr )= 
-\langle \mathrm v_2, \mathrm w_1\rangle + \langle \mathrm v_1,\mathrm w_2\rangle, 
\end{equation}
where $\mathrm v,\mathrm w\in V=\Om^1(ad\, P)$. As already mentioned above, the moment map of the real gauge group $\mathcal G(P)$ is given by $(a,b)\mapsto \delta_a b$. Moreover,  $\mathcal A_{asd}(\mathcal P)=\{ A\in\mathcal A(\mathcal P)\mid F_A^+=0 \}$ is a complex subvariety of the configuration space. Then  $\Mcasd (\mathcal P)$ is the K\"ahler reduction of  $\mathcal A_{asd}(\mathcal P)$ with respect to the action of the real gauge group, hence a K\"ahler manifold.

To see \textit{(ii)}, observe that the antisymplectic involution $(a,b)\mapsto (a,-b)$ on the configuration space induces an antisymplectic involution $\tau$ on $\Mcasd (\mathcal P)$. The fixed point set of $\tau$ is $\mathcal M_{asd}(P)$, whose dimension equals $\frac 12\dim \Mcasd (\mathcal P)$. Hence,  $\mathcal M_{asd}(P)$ is a Lagrangian submanifold of $\Mcasd (\mathcal P)$.

It remains to show~\textit{(iii)}. Recall that for a K\"ahler surface $X$ there is the decomposition $\Om^2_+(X;\R)\cong \Om^0(X)\cdot \om_X\oplus \Om^{0,2}(X)$. Denote by $\Lambda\colon\Om^2(X)\to \Om^0(X)$ the adjoint operator of $L\colon\Om^0(X)\to \Om^2(X),\ \alpha\mapsto \alpha\om_X$. Then~\eqref{Eq_cxASD1}-\eqref{Eq_cxASD3} can be written in the form
\begin{align}
  &\delta_ab=0, \qquad\Lambda d_a^+b=0,\quad \Lambda (F_a^+-\tfrac 12 [b\wedge b]^+)=0,\label{Eq_casdOnKaehler1}\\ 
 & F_{a+bi}^{2,0}=0, \quad F_{a+bi}^{0,2}=0.\label{Eq_casdOnKaehler2}
\end{align}

Furthermore, for a Hermitian vector space $\bigl (V, \langle\cdot,\cdot\rangle+ i\om(\cdot,\cdot)\bigr )$ its complexification $V_\C\cong V\oplus V$ is a quaternion--Hermitian vector space. Explicitly, the quaternion--Hermitian structure is given by~\eqref{Eq_KaehlerStr} together with
\[
\begin{aligned}
  &I_2 (\mathrm v,\mathrm w)=(I\mathrm v,-I\mathrm w), &\qquad & 
\om_2\bigl ((\mathrm v_1, \mathrm w_1), (\mathrm v_2, \mathrm w_2)\bigr )= 
\om ( \mathrm v_1, \mathrm v_2) -\om( \mathrm w_1,\mathrm w_2), \\
&I_3 (\mathrm v,\mathrm w)=(I\mathrm w, I\mathrm v), &\qquad & 
\om_3\bigl ((\mathrm v_1, \mathrm w_1), (\mathrm v_2, \mathrm w_2)\bigr )= 
\om ( \mathrm w_1, \mathrm v_2 ) + \om( \mathrm v_1,\mathrm w_2). 
\end{aligned}
\]
This implies that $\mathcal A(\mathcal P)\cong \mathcal A(P)\times\Om^1(ad\, P)$ is a flat \hK manifold. A straightforward but somewhat lengthy computation shows that the action of the real gauge group $\mathcal G (P)$ preserves this \hK structure and the zero level set of the corresponding \hK moment map is given by solutions of~\eqref{Eq_casdOnKaehler1}. Furthermore, denote by $\mathcal A^{1,1}(\mathcal P)$ the space of solutions of~\eqref{Eq_casdOnKaehler2}. Clearly, the tangent bundle of $\mathcal A^{1,1}(\mathcal P)$ is preserved by $I_1$ and $I_2$. Therefore,  $\mathcal A^{1,1}(\mathcal P)$ is a \hK submanifold of $\mathcal A(\mathcal P)$. Thus, for a K\"ahler surface $X$ the moduli space $\Mcasd(\mathcal P)$ is the \hK reduction of $\mathcal A^{1,1}(\mathcal P)$, hence a \hK manifold. The remaining part of \textit{(iii)} is shown in a similar manner to~\textit{(ii)}.

\subsection{$Spin(7)$--instantons}\label{Subsect_Spin7Instantons}

At first, it is convenient to recall the notion of a $Spin(7)$--instanton, which appeared in the mathematical literature in~\cite{DonaldsonThomas:98} for the first time. To do this, fix a splitting $\R^8=U\oplus V$, where $U\cong\H\cong V$. Let $\theta$ (resp. $\eta$) denote the projection onto the first (resp. second) subspace. Think of $\theta$ and $\eta$ as $\H$--valued 1--forms on $\R^8$. The stabilizer of the Cayley form
\[
\Om=-\frac 1{24}\re{\Bigl ( \theta\wedge\bar\theta\wedge \theta\wedge\bar\theta 
 -6\,  \theta\wedge\bar\theta\wedge  \eta\wedge\bar\eta
+ \eta\wedge\bar\eta\wedge  \eta\wedge\bar\eta \Bigr )},
\]  
 is~\cite{BryantSalamon:89} the subgroup $Spin(7)\subset SO(8)$. The Cayley form gives rise to the linear map
\[
\Lambda^2(\R^8)^*\rightarrow \Lambda^2(\R^8)^*,\qquad \om\mapsto -*(\om\wedge \Om),
\]
which has two eigenvalues $3$ and $-1$. The corresponding eigenspaces $\Lambda^2_+(\R^8)^*$ and $\Lambda^2_-(\R^8)^*$ are of dimensions $7$ and $21$ respectively. Hence, just like in 4 dimensions, there is the decomposition of the space of 2-forms:
\begin{equation}\label{Eq_DecompOf2FormsInDim8}
  \Om^2(\R^8)= \Om^2_+(\R^8)\oplus \Om^2_-(\R^8).
\end{equation}
Then a connection $A$ on a $G$--bundle over $\R^8$ is called a $Spin(7)$--instanton, if $F_A^+=0$.
\begin{rem}
  In general, the base manifold $\R^8$ can be replaced by a Riemannian eight-manifold with holonomy $Spin(7)$. 
\end{rem}

Notice also that the splitting $\R^8=U\oplus V$ leads to the decomposition 
\[
\Lambda^k(\R^8)^*= \bigoplus_{p+q=k}\Lambda^{p,q}(\R^8)^*,\qquad\text{where}\quad
\Lambda^{p,q}(\R^8)^*\cong \Lambda^p U^*\otimes\Lambda^q V^*.
\]
Hence, there is a similar decomposition of differential forms on $\R^8$:
\[
\Om^k(\R^8)=\bigoplus_{p+q=k}\Om^{p,q}(\R^8).
\]
In particular, any connection on the trivial bundle $\underline{G}\to\R^8$ can be uniquely written as $A=a+b$, where $a\in\Om^{1,0}(\R^8;\mathfrak g)$ and $b\in\Om^{0,1}(\R^8;\mathfrak g)$. Think of $b$ as a family of connections on $\R^4\cong V\subset \R^8$ parametrized by $U$. Let $F_b$ denote the corresponding family of curvatures. Then a computation~\cite{Haydys:12_GaugeTheory_jlms} shows that $A$ is a $Spin(7)$--instanton if and only if
\[
(F_A^{1,1})^+=0\quad\text{and}\quad F_a^+=F_b^+.
\]
Notice that the superscript ``+'' in the first equation is used in the sense of decomposition~\eqref{Eq_DecompOf2FormsInDim8}, while in the second one in the sense of the four-dimensional analogue of~\eqref{Eq_DecompOf2FormsInDim8}. 

\medskip

Putting aside $Spin(7)$--instantons for a while, consider the \gSW equations for the target \hK manifold $\Om^1(\R^4;\mathfrak g)$, which is interpreted as the space of connections on the trivial bundle $\underline G\to \R^4$. The gauge group $Map(\R^4; G)$ acts on $\Om^1(\R^4;\mathfrak g)$ preserving its \hK structure. The corresponding moment map is well-known to be $\mu(b)=F_b^+$. Hence, in this case the \gSW equations can be written as
\[
\dirac_a b=0,\qquad F_a^+=F_b^+.
\] 
A somewhat lengthy computation, which can be found in details in~\cite{Haydys:12_GaugeTheory_jlms}, shows that $\dirac_a b=(F_A^{1,1})^+$, where $A=a+b$. Thus, for $M=\Om^1(\R^4;\mathfrak g)$ the \gSW equations yield $Spin(7)$--instantons.

\begin{rem}
If $\R^4$ is replaced by a spin four-manifold $X$, the \gSW equations with the target manifold  $M=\Om^1(\R^4;\mathfrak g)$ yield up to a zero-order term $Spin(7)$--instantons on the total space of the spinor bundle of $X$~\cite{Haydys:12_GaugeTheory_jlms}.
\end{rem}

\begin{rem}\label{Rem_G2Monopoles}
By using similar arguments, one can also show that solutions of the \gSW equations with the target manifold  $M=\Om^1(\R^3;\mathfrak g)\times\Om^0(\R^3;\mathfrak g)$ can be interpreted as $G_2$--monopoles on $\R^7$.
\end{rem}

\subsection{Five-dimensional instantons}\label{Subsect_5dInstantons}

Consider $M=\Om^0(\R;\,\mathfrak g\otimes\H)$ as the target manifold equipped with its flat \hK structure, where $\mathfrak g$ is the Lie algebra of a Lie group $G$. The gauge group $\mathcal G=Map(\R; G)$ acts on $\Om^0(\R;\,\mathfrak g\otimes\H)$, namely $g\cdot T = gTg^{-1}-\dot g g^{-1}$, where $\dot g$ is the derivative of $g$ with respect to the variable $t\in\R$. Then the \hK moment map $\mu=\mu_1i+\mu_2j+\mu_3k$ of this action is given by
\[
\begin{aligned}
  & \mu_1(T)=\dot T_1 + [T_0, T_1] -[T_2, T_3],\\
  & \mu_2(T)=\dot T_2 + [T_0, T_2] -[T_3, T_1],\\
  & \mu_3(T)=\dot T_3 + [T_0, T_3] -[T_1, T_2],
\end{aligned}
\]
where $T=T_0 + T_1i + T_2j+ T_3k$. 
Recalling the identification $\H=\R\oplus\Lambda^2_+(\R^4)^*$, a map $u\colon\R^4\to \Om^0(\R;\,\mathfrak g\otimes\H)$ can be interpreted as a map  $(c,b)\colon\R\to  \Om^0(\R^4;\,\mathfrak g)\times  \Om^2_+(\R^4;\,\mathfrak g)$. Similarly, a connection on $\underline{\mathcal G}\to \R^4$ can be interpreted as a map $a\colon\R\to\Om^1(\R^4;\mathfrak g)$. A little thought shows that the corresponding \gSW equations can be written in the form
\begin{equation}\label{Eq_5dInstantons}
\begin{aligned}
  &\dot a=\delta_a b + d_a c,\\
  &\dot b= F_a^+-\sigma (b) - [c,b].
\end{aligned}
\end{equation}
These are the five-dimensional instantons on $W^5=\R^4\times\R$~\cite[Eqns. (40)]{Haydys:10_FukayaSeidelCategory}. Equations~\eqref{Eq_5dInstantons} were discovered independently by Witten~\cite{Witten:12_FivebranesAndKnots_QT}. The author of this paper actually obtained five-dimensional instantons for the first time along the lines outlined above.  The reader can find some applications of five-dimensional instantons in the papers mentioned above.

\begin{rem}
  One can replace $\Om^0(\R;\,\mathfrak g\otimes\H)$ in the above construction  by the space $\tilde \Om^0(I;\,\mathfrak g\otimes\H)$ consisting of all maps $T\colon I=(-1, 1)\to\mathfrak g\otimes\H$ satisfying certain asymptotic conditions as $t\to\pm 1$. Then the \hK reduction of $\tilde \Om^0(I;\,\mathfrak g\otimes\H)$ is the moduli space of magnetic monopoles on $\R^3$. This is also the \hK reduction of the target manifold from Remark~\ref{Rem_G2Monopoles}. Hence, it is naturally to expect that there is some relation between five-dimensional instantons and $G_2$-monopoles. 
\end{rem}


\section{Remarks on three-manifolds}

As already mentioned above, there is a generalization of the Dirac operator for three-manifolds due to Taubes~\cite{Taubes:99}. Hence, one can also consider the \gSW equations in dimension three with the target manifolds as in the preceding section. It turns out that  the analogues of the Vafa--Witten equations and the complex anti-self-duality equations coincide and yield stable flat $G^\C$--connections. The compactness property of the moduli space of flat $PSL(2;\C)$--connections was recently studied by Taubes~\cite{Taubes12_PSL2Cconnections_ArxPrepr}. Finally, the construction of Subsection~\ref{Subsect_Spin7Instantons} leads to $G_2$--instantons, while that of Subsection~\ref{Subsect_5dInstantons} leads to the Kapustin--Witten equations~\cite{KapustinWitten:07_ElectrMagnDuality}.







\bibliographystyle{alphanum}

\def\cprime{$'$}

\end{document}